\documentclass{amsart}

\usepackage{amssymb,amsmath,amsthm,latexsym,booktabs}

\theoremstyle{definition}
\newtheorem{definition}{Definition}

\newtheorem{example}[definition]{Example}

\newtheorem{remark}[definition]{Remark}

\theoremstyle{plain}
\newtheorem{lemma}[definition]{Lemma}
\newtheorem{proposition}[definition]{Proposition}
\newtheorem{theorem}[definition]{Theorem}
\newtheorem{corollary}[definition]{Corollary}

\newcommand{\nn}{\!\!\!\!\!\!}

\begin{document}

\title{Malcev dialgebras}

\author[Bremner]{Murray R. Bremner}

\address{Department of Mathematics and Statistics,
University of Saskatchewan, Canada}

\email{bremner@math.usask.ca}

\author[Peresi]{Luiz A. Peresi}

\address{Department of Mathematics,
University of S\~ao Paulo, Brazil}

\email{peresi@ime.usp.br}

\author[S\'anchez-Ortega]{Juana S\'anchez-Ortega}

\address{Department of Algebra, Geometry and Topology, University of M\'alaga, Spain}

\email{jsanchezo@uma.es}

\begin{abstract}
We apply Kolesnikov's algorithm to obtain a variety of nonassociative algebras defined by 
right anticommutativity and a ``noncommutative'' version of the Malcev identity.
We use computational linear algebra to verify that these identities are equivalent to 
the identities of degree $\le 4$ satisfied by the dicommutator in every alternative dialgebra.
We extend these computations to show that any special identity for Malcev dialgebras must have degree at least 7.
Finally, we introduce a trilinear operation which makes any Malcev dialgebra into a Leibniz triple system.
\end{abstract}

\maketitle


\section{Introduction}

In this paper we introduce the appropriate generalization of Malcev algebras to the setting of dialgebras.
These new structures, which we call Malcev dialgebras, are related to Malcev algebras 
in the same way that Leibniz algebras are related to Lie algebras; 
they are related to alternative dialgebras in the same way that Malcev algebras are related to alternative algebras.
To obtain the defining identities for Malcev dialgebras, we apply Kolesnikov's algorithm 
to anticommutativity and the Malcev identity; 
we obtain right anticommutativity and a ``noncommutative'' version of the Malcev identity.
We use computer algebra to verify that these identities are equivalent to the identities of degree $\le 4$ 
satisfied by the dicommutator in every alternative dialgebra,
and that the resulting identities imply every identity of degree $\le 6$ satisfied by the dicommutator 
in every alternative dialgebra. 
We then generalize the construction of Loos, which defines the structure of a Lie triple system on a Malcev algebra, 
to the setting of dialgebras: 
we introduce a trilinear operation on a Malcev dialgebra which makes the underlying vector space into 
a Leibniz triple system in the sense of Bremner and S\'anchez-Ortega \cite{BSO}.


\section{Preliminaries} \label{sectionpreliminaries}

Dialgebras were introduced by Loday \cite{LodayDialgebras} (see also \cite{LodaySurvey}) 
to provide a natural setting for Leibniz algebras, a ``noncommutative'' version 
of Lie algebras.

\begin{definition} (Cuvier \cite{Cuvier}, Loday \cite{LodayLeibniz})
A {\bf Leibniz algebra} is a vector space $L$ together with a bilinear map 
$L \times L \to L$, denoted $(a,b) \mapsto \langle a,b \rangle $, 
satisfying the {\bf Leibniz identity}, which says that right multiplications are derivations:
  \begin{equation} \label{Rightleibnizidentity}
  \langle \langle a, b \rangle, c\rangle 
  \equiv 
  \langle \langle a, c \rangle, b \rangle 
  + 
  \langle a, \langle b, c   \rangle \rangle.
  \end{equation}
If $\langle a, a \rangle \equiv 0$ then the Leibniz identity is the Jacobi identity and $L$
is a Lie algebra.
\end{definition}

An associative algebra becomes a Lie algebra if the associative product is replaced by the Lie bracket. 
The notion of dialgebra gives, by a similar procedure, a Leibniz algebra: one replaces the associative 
products $ab$ and $ba$ by two distinct operations, so that the resulting bracket is not 
necessarily skew-symmetric.

\begin{definition} 
\label{definitiondialgebras}
(Loday \cite{LodayDialgebras})
A \textbf{dialgebra} is a vector space $D$ with two bilinear operations
$\dashv\colon D \times D \to D$ and $\vdash\colon D \times D \to D$,
called the \textbf{left} and \textbf{right} products.
\end{definition} 

\begin{definition} 
\label{definition0dialgebras}
(Kolesnikov \cite{Kolesnikov})
A \textbf{0-dialgebra} is a dialgebra satisfying the \textbf{left} and \textbf{right bar identities}:
  \[ 
  ( a \dashv b ) \vdash c
  \equiv
  ( a \vdash b ) \vdash c,
  \qquad
  a \dashv ( b \dashv c )
  \equiv
  a \dashv ( b \vdash c ).
  \]
\end{definition} 

\begin{definition} 
\label{definitionassociativedialgebras}
(Loday \cite{LodayDialgebras})
An \textbf{associative dialgebra} is a 0-dialgebra satisfying \textbf{left, right} and \textbf{inner associativity}:
  \[
  ( a \dashv b ) \dashv c
  \equiv
  a \dashv ( b \dashv c ),
  \quad
  ( a \vdash b ) \vdash c
  \equiv
  a \vdash ( b \vdash c ),
  \quad
  ( a \vdash b ) \dashv c
  \equiv
  a \vdash ( b \dashv c ).
  \]
\end{definition} 

\begin{definition}
In any dialgebra, the \textbf{dicommutator} is the bilinear operation 
  \[
  \langle a, b \rangle = a \dashv b - b \vdash a.
  \]
\end{definition}

An associative dialgebra gives rise to a Leibniz algebra by considering the same underlying vector space 
with the product defined to be the dicommutator.
The goal of the present paper is to study the same construction for alternative dialgebras.

\begin{definition} 
\label{definitionalternativedialgebras}
(Liu \cite{Liu})
An \textbf{alternative dialgebra} is a 0-dialgebra satisfying:
  \[
  (a,b,c)_\dashv + (c,b,a)_\vdash \equiv 0,
  \quad
  (a,b,c)_\dashv - (b,c,a)_\vdash \equiv 0,
  \quad
  (a,b,c)_\times + (a,c,b)_\vdash \equiv 0,
  \]
where the \textbf{left}, \textbf{right} and \textbf{inner associators} are defined by
\allowdisplaybreaks 
\begin{alignat*}{2}
  (a,b,c)_\dashv &= ( a \dashv b ) \dashv c - a \dashv ( b \dashv c ),
  &\qquad
  (a,b,c)_\vdash &= ( a \vdash b ) \vdash c - a \vdash ( b \vdash c ),
  \\
  (a,b,c)_\times &= ( a \vdash b ) \dashv c - a \vdash ( b \dashv c ).
  \end{alignat*}
\end{definition}

\subsection*{Kolesnikov's algorithm}

Kolesnikov \cite{Kolesnikov} (see also Pozhidaev \cite{Pozhidaev}) introduced a general framework for 
converting the defining identities of a variety of algebras into the defining identities of 
the corresponding variety of dialgebras.
Part 1 converts a multilinear polynomial identity of degree $d$ 
for a bilinear operation into $d$ multilinear identities of degree $d$ for two new bilinear operations. 
Part 2 introduces the analogues of the bar identities for each new operation.

Part 1:
We consider a bilinear operation, not necessarily associative, denoted by the symbol $\{-,-\}$.
Given a multilinear polynomial identity of degree $d$ in this operation, 
we show how to apply the algorithm to one monomial, 
and from this the application to the complete identity follows by linearity. 
Let $\overline{a_1 a_2 \dots a_d}$ be a multilinear monomial of degree $d$, 
where the bar denotes some placement of operation symbols. 
We introduce two new operations, denoted by the same symbol but distinguished by subscripts: $\{-,-\}_1$, $\{-,-\}_2$.
For each $i \in \{1, 2, \dots, d\}$ we convert the monomial $\overline{a_1 a_2 \dots a_d}$
in the original operation into a new monomial of the same degree $d$ in the two new operations, 
according to the following rule, which is based on the position of $a_i$,
called the central argument of the monomial. 
For each occurrence of the original operation $\{-,-\}$ in the monomial, 
either $a_i$ occurs within one of the two arguments or not, and we have the following cases:
  \begin{itemize}
  \item
  If $a_i$ occurs within the $j$-th argument then we convert the original operation
  $\{-,-\}$ to the $j$-th new operation $\{-,-\}_j$.
  \item
  If $a_i$ does not occur within either of the two arguments, then either
    \begin{itemize}
    \item
    $a_i$ occurs to the left of the original operation, 
    in which case we convert $\{-,-\}$ to the first new operation $\{-,-\}_1$, 
    or
    \item
    $a_i$ occurs to the right of the original operation,
    in which case we convert $\{-,-\}$ to the second new operation $\{-,-\}_2$.
    \end{itemize}
  \end{itemize}

Part 2:
We also include the following two identities, analogous to the left and 
right bar identities of Definition \ref{definition0dialgebras}.
These identities say that the two new operations are interchangeable in 
the $i$-th argument of the $j$-th new operation when $i \ne j$:
  \[
  \{ \{ a, b \}_1, c \}_2 \equiv \{ \{ a, b \}_2, c \}_2,
  \qquad
  \{ a, \{ b, c \}_1 \}_1 \equiv \{ a, \{ b, c \}_2 \}_1.
  \]

\begin{example}
The definition of associative dialgebra can be obtained by applying Kolesnikov's algorithm 
to associativity, $\{ \{ a, b \}, c \} \equiv \{ a, \{ b, c \} \}$.
Part 1 produces three new identities of degree 3
by making $a$, $b$, $c$ in turn the central argument:
  \begin{alignat*}{2}
  &
  \{ \{ a, b \}_1, c \}_1 \equiv \{ a, \{ b, c \}_1 \}_1, 
  &\quad
  &
  \{ \{ a, b \}_2, c \}_1 \equiv \{ a, \{ b, c \}_1 \}_2,
  \\
  &
  \{ \{ a, b \}_2, c \}_2 \equiv \{ a, \{ b, c \}_2 \}_2.
  \end{alignat*} 
Combining these identities with the two identities from Part 2, and reverting to the standard notation 
$a \dashv b = \{ a, b \}_1$, $a \vdash b = \{ a, b \}_2$,
we obtain Definition \ref{definitionassociativedialgebras}.
\end{example}

\begin{example}
The definition of alternative dialgebra can be obtained by applying Kolesnikov's algorithm 
to right and left alternativity, $(a,a,b) \equiv 0$ and $(b,a,a) \equiv 0$,
where $(x,y,z)= (xy)z - x(yz)$ is the associator. 
If we assume characteristic not 2, then these two identities are equivalent to 
their multilinear forms; we expand the associators and use the operation symbol $\{-,-\}$:
  \begin{align*}  
  &
  \{\{a,b\},c\} - \{a,\{b,c\}\} + \{\{b,a\},c\} - \{b,\{a,c\}\} \equiv 0, \\ 
  &
  \{\{a,b\},c\} - \{a,\{b,c\}\} + \{\{a,c\},b\} - \{a,\{c,b\}\} \equiv 0.
  \end{align*} 
Part 1 produces six identities relating the two new operations $\{-,-\}_1$ and $\{-,-\}_2$:
 \allowdisplaybreaks
  \begin{alignat}{2}
  &
  \{\{a,b\}_1,c\}_1 - \{a,\{b,c\}_1\}_1 + \{\{b,a\}_2,c\}_1 - \{b,\{a,c\}_1\}_2 \equiv 0,
  \label{Id11}
  \\
  &
  \{\{a,b\}_2,c\}_1 - \{a,\{b,c\}_1\}_2 + \{\{b,a\}_1,c\}_1 - \{b,\{a,c\}_1\}_1 \equiv 0,
  \label{Id12}
  \\   
  &
  \{\{a,b\}_2,c\}_2 - \{a,\{b,c\}_2\}_2 + \{\{b,a\}_2,c\}_2 - \{b,\{a,c\}_2\}_2 \equiv 0,
  \label{Id13}
  \\
  &
  \{\{a,b\}_1,c\}_1 - \{a,\{b,c\}_1\}_1 + \{\{a,c\}_1,b\}_1 - \{a,\{c,b\}_1\}_1 \equiv 0,
  \label{Id21}
  \\
  &
  \{\{a,b\}_2,c\}_1 - \{a,\{b,c\}_1\}_2 + \{\{a,c\}_2,b\}_2 - \{a,\{c,b\}_2\}_2 \equiv 0,
  \label{Id22}
  \\
  &
  \{\{a,b\}_2,c\}_2 - \{a,\{b,c\}_2\}_2 + \{\{a,c\}_2,b\}_1 - \{a,\{c,b\}_1\}_2 \equiv 0.
  \label{Id23}
  \end{alignat}
Writing $a\dashv b = \{a,b\}_1$ and $a\vdash b = \{a,b\}_2$ and using the dialgebra associators
of Definition \ref{definitionalternativedialgebras}, 
we rewrite identities (\ref{Id11}) to (\ref{Id23}) as follows:
  \allowdisplaybreaks
  \begin{alignat*}{2}
  (a,b,c)_\dashv + (b,a,c)_\times 
  &\equiv 0,
  &\qquad
  (a,b,c)_\times + (b,a,c)_\dashv 
  &\equiv 0,
  \\
  (a,b,c)_\vdash + (b,a,c)_\vdash 
  &\equiv 0,
  &\qquad
  (a,b,c)_\dashv + (a,c,b)_\dashv 
  &\equiv 0,
  \\
  (a,b,c)_\times + (a,c,b)_\vdash 
  &\equiv 0,
  &\qquad
  (a,b,c)_\vdash + (a,c,b)_\times 
  &\equiv 0.
  \end{alignat*}
These identities show how the transpositions $(ab)$ and $(bc)$ affect the dialgebra associators. 
It is now clear that (\ref{Id11}) and (\ref{Id12}) are equivalent, as are (\ref{Id22}) and (\ref{Id23}).  
Furthermore, (\ref{Id13}) can be derived from (\ref{Id11}), (\ref{Id21}) and (\ref{Id22}):
  \[
  (a,b,c)_\vdash
  \stackrel{(\ref{Id22})}{\equiv}
  -
  (a,c,b)_\times
  \stackrel{(\ref{Id11})}{\equiv}
  (c,a,b)_\dashv
  \stackrel{(\ref{Id21})}{\equiv}
  -
  (c,b,a)_\dashv
  \stackrel{(\ref{Id11})}{\equiv}
  (b,c,a)_\times
  \stackrel{(\ref{Id22})}{\equiv}
  -
  (b,a,c)_\vdash.
  \]
Thus the final result is the variety of dialgebras satisfying these identities: 
  \begin{equation}
  \label{alternativedialgebraidentities}
  (a,b,c)_\dashv + (b,a,c)_\times \equiv 0,
  \;
  (a,b,c)_\dashv + (a,c,b)_\dashv \equiv 0,
  \;
  (a,b,c)_\times + (a,c,b)_\vdash \equiv 0,
  \end{equation}
which are equivalent to the identities of Definition \ref{definitionalternativedialgebras}.
\end{example}

\begin{remark}
The algorithm of Kolesnikov is closely related to general constructions in the theory of operads
discussed by Chapoton \cite{Chapoton} and Vallette \cite{Vallette}.
\end{remark}


\section{The definition of Malcev dialgebra} \label{sectionMalcevdialgebras}

In this section we recall the defining identities for Malcev algebras, 
and then apply Kolesnikov's algorithm to obtain the defining identities for 
the corresponding variety of dialgebras, which we call Malcev dialgebras. 

\begin{definition} 
(Malcev \cite{Malcev})
A \textbf{Malcev algebra} is a vector space with a bilinear operation $ab$ satisfying 
\textbf{anticommutativity} and the \textbf{Malcev identity}:
  \[
  a^2 \equiv 0,
  \qquad
  J( a, b, ac ) \equiv J( a, b, c ) a,
  \]
where $J(a,b,c) = (ab)c + (bc)a + (ca)b$ is the Jacobian.  
\end{definition}

\begin{lemma}
\label{malcevdefinition}
\emph{(Sagle \cite{Sagle})}
If the characteristic is not 2, then an algebra is Malcev if and only if 
it satisfies the following multilinear identities:
  \[ 
  ab + ba \equiv 0,
  \qquad
  (ac)(bd) \equiv ((ab)c)d + ((bc)d)a + ((cd)a)b + ((da)b)c.
  \]
\end{lemma}

To apply Kolesnikov's algorithm, we write these identities as follows:
  \begin{align*}
  & 
  \{a,b\} + \{b,a\} \equiv 0, 
  \\
  & 
  \{\{a,c\},\{b,d\}\} - \{\{\{a,b\},c\},d\} - \{\{\{b,c\},d\},a\} - \{\{\{c,d\},a\},b\}
  \\
  &\qquad
  - \{\{\{d,a\},b\},c\} \equiv 0.
\end{align*}
Part 1 gives six identities relating the operations $\{-,-\}_1$ and $\{-,-\}_2$: 
  \allowdisplaybreaks
  \begin{align}
  & 
  \begin{array}{c}
  \{a,b\}_1 + \{b,a\}_2 \equiv 0, 
  \end{array}
  \quad
  \begin{array}{c}
  \{a,b\}_2 + \{b,a\}_1 \equiv 0,
  \end{array}
  \label{anticomm}
  \\
  & 
  \begin{array}{c}
  \{\{a,c\}_1,\{b,d\}_1\}_1 - \{\{\{a,b\}_1,c\}_1,d\}_1 - \{\{\{b,c\}_2,d\}_2,a\}_2
  \\[3pt]
  \qquad 
  - \{\{\{c,d\}_2,a\}_2,b\}_1 - \{\{\{d,a\}_2,b\}_1,c\}_1 \equiv 0,
  \end{array}
  \label{malcev1} 
  \\
  & 
  \begin{array}{c}
  \{\{a,c\}_2,\{b,d\}_1\}_2 - \{\{\{a,b\}_2,c\}_1,d\}_1 - \{\{\{b,c\}_1,d\}_1,a\}_1
  \\[3pt]
  \qquad 
  - \{\{\{c,d\}_2,a\}_2,b\}_2 - \{\{\{d,a\}_2,b\}_2,c\}_1 \equiv 0, 
  \end{array}
  \label{malcev2}
  \\
  & 
  \begin{array}{c}
  \{\{a,c\}_2,\{b,d\}_1\}_1 - \{\{\{a,b\}_2,c\}_2,d\}_1 - \{\{\{b,c\}_2,d\}_1,a\}_1
  \\[3pt]
  \qquad 
  - \{\{\{c,d\}_1,a\}_1,b\}_1 - \{\{\{d,a\}_2,b\}_2,c\}_2 \equiv 0,
  \end{array}
  \label{malcev3}
  \\
  & 
  \begin{array}{c}
  \{\{a,c\}_2,\{b,d\}_2\}_2 - \{\{\{a,b\}_2,c\}_2,d\}_2 - \{\{\{b,c\}_2,d\}_2,a\}_1 
  \\[3pt]
  \qquad 
  - \{\{\{c,d\}_2,a\}_1,b\}_1 - \{\{\{d,a\}_1,b\}_1,c\}_1 \equiv 0.
  \end{array}
  \label{malcev4}
  \end{align}
The two identities \eqref{anticomm} are equivalent;
both say $\{a,b\}_2 \equiv - \{b,a\}_1$, so we can eliminate the second operation. 
Applying this to identities (\ref{malcev1})--(\ref{malcev4}), we obtain:
  \allowdisplaybreaks
  \begin{align}
  & 
  \begin{array}{c}
  \{\{a,c\}_1,\{b,d\}_1\}_1 - \{\{\{a,b\}_1,c\}_1,d\}_1 + \{a,\{d,\{c,b\}_1\}_1\}_1 
  \\[3pt]
  \qquad
  - \{\{a,\{d,c\}_1\}_1,b\}_1 + \{\{\{a,d\}_1,b\}_1,c\}_1 \equiv 0,
  \end{array}
  \label{malcev11} 
  \\
  & 
  \begin{array}{c}
  \{\{b,d\}_1,\{c,a\}_1\}_1 + \{\{\{b,a\}_1,c\}_1,d\}_1 - \{\{\{b,c\}_1,d\}_1,a\}_1
  \\[3pt]
  \qquad 
  + \{b,\{a,\{d,c\}_1\}_1\}_1 - \{\{b,\{a,d\}_1\}_1,c\}_1 \equiv 0, 
  \end{array}
  \label{malcev21}
  \\
  & 
  \begin{array}{c}
  \{\{c,a\}_1,\{b,d\}_1\}_1 + \{\{c,\{b,a\}_1\}_1,d\}_1 - \{\{\{c,b\}_1,d\}_1,a\}_1
  \\[3pt]
  \qquad 
  + \{\{\{c,d\}_1,a\}_1,b\}_1 - \{c,\{b,\{a,d\}_1\}_1\}_1 \equiv 0,
  \end{array}
  \label{malcev31}
  \\
  & 
  \begin{array}{c}
  \{\{d,b\}_1,\{c,a\}_1\}_1 - \{d,\{c,\{b,a\}_1\}_1\}_1 + \{\{d,\{c,b\}_1\}_1,a\}_1
  \\[3pt]
  \qquad
  - \{\{\{d,c\}_1,a\}_1,b\}_1 + \{\{\{d,a\}_1,b\}_1,c\}_1 \equiv 0.
  \end{array}
  \label{malcev41}
  \end{align}
Since we now have only one operation, we revert to a simpler notation, 
and write $\{a,b\}_1$ simply as $ab$.
Identities \eqref{malcev11}--\eqref{malcev41} take the following form:
  \allowdisplaybreaks
  \begin{align}
  & 
  (ac)(bd) - ((ab)c)d + a(d(cb)) - (a(dc))b + ((ad)b)c \equiv 0,
  \label{malcev12} 
  \\
  & 
  (bd)(ca) + ((ba)c)d - ((bc)d)a + b(a(dc)) - (b(ad))c \equiv 0,
  \label{malcev22}  
  \\
  - \;
  & 
  (ca)(bd) - (c(ba))d + ((cb)d)a - ((cd)a)b + c(b(ad)) \equiv 0,
  \label{malcev32}
  \\
  - \;
  & 
  (db)(ca) + d(c(ba)) - (d(cb))a + ((dc)a)b - ((da)b)c \equiv 0.
  \label{malcev42}
  \end{align}
Part 2 gives two identities; rewriting them in terms of the first operation gives
  \[
  \{ a, \{ b, c \}_1 \}_1 \equiv - \{ a, \{ c, b \}_1 \}_1,
  \qquad
  - \{ c, \{ a, b \}_1 \}_1 \equiv \{ c, \{ b, a \}_1 \}_1,
  \]
which are both equivalent to right anticommutativity $a(bc) \equiv - a(cb)$.
We note that \eqref{malcev42} is a permutation of \eqref{malcev32}. 
Furthermore, rearranging the terms in \eqref{malcev22} and \eqref{malcev32}, 
and applying right anticommutativity, gives \eqref{malcev12}.
Thus we require only one identity in degree 4.

\begin{definition}  \label{MDdefinition}
Over a field of characteristic not $2$, a \textbf{(right) Malcev dialgebra} 
is a vector space with a bilinear operation $ab$
satisfying \textbf{right anticommutativity} and the \textbf{di-Malcev identity}:
  \[
  a(bc) + a(cb) \equiv 0,
  \qquad
  ((ab)c)d - ((ad)b)c - (a(cd))b - (ac)(bd) - a((bc)d) \equiv 0.
  \]
\end{definition}


\section{The dicommutator in an alternative dialgebra} \label{dicommutatorsection}

Malcev \cite{Malcev} showed that an alternative algebra becomes a Malcev algebra 
by considering the same underlying vector space with the new operation $ab - ba$.
In this section we extend this result to the setting of dialgebras:
we use computer algebra to show that any subspace of an alternative dialgebra closed under the dicommutator 
is a Malcev dialgebra, and conversely that the dicommutator identities of degrees $\le 4$ 
are equivalent to right anticommutativity and the di-Malcev identity. 

We write $F\!A_n$ for the multilinear subspace of degree $n$ in the free nonassociative algebra 
on $n$ generators.  The number of association types (distinct placements of parentheses) in 
degree $n$ is the Catalan number,
  \[
  K_n = \frac{1}{n} \binom{2n{-}2}{n{-}1}.
  \]
Since there are $n!$ permutations of $n$ indeterminates, we have $\dim F\!A_n = n! K_n$.

\begin{lemma} \label{barsidelemma}
Let $X$ be a set. For any $a_1, \hdots, a_n \in X$, let $w = a_1 \hdots a_n$ be a monomial 
in the free 0-dialgebra on $X$, with some placement of parentheses and choice of operations. 
If $x \dashv y$ or $y \vdash x$ is a submonomial, then $y$ does not depend on the choice of 
operations: we may regard $y$ as a monomial in the free algebra.
\end{lemma}

\begin{proof}
Induction on the degree $n$ using the bar identities of Definition \ref{definition0dialgebras}.
\end{proof}

We write $FD_n$ for the multilinear subspace of degree $n$ in the free 0-dialgebra on $n$ generators.

\begin{lemma} \label{dialgebratypeslemma}
The number of 0-dialgebra association types in degree $n$ is
  \[
  Z_n = \binom{2n{-}2}{n{-}1}.
  \]
\end{lemma}

\begin{proof}
Suppose that we have enumerated the 0-dialgebra association types up to degree $n{-}1$. 
By Lemma \ref{barsidelemma}, any 0-dialgebra association type in degree $n$ is either 
$x \dashv y$ or $y \vdash x$ where $x$ is a 0-dialgebra association type in degree $n{-}i$ 
and $y$ is an algebra association type in degree $i$, for some $i < n$.  Therefore
  \[
  Z_1 = 1,
  \qquad
  Z_n = 2 \sum_{i=1}^{n-1} Z_{n-i} K_i \;\; (n \ge 2).
  \]
The unique solution to this recurrence relation is $Z_n = n K_n$.
\end{proof}

\begin{lemma} \label{dialgebradimensionlemma}
We have $\dim FD_n = n! Z_n$.
\end{lemma}

\begin{proposition} \label{degree3proposition}
Over a field of characteristic not 2 or 3, every multilinear polynomial
identity in degree 3 satisfied by the dicommutator in every alternative
dialgebra is a consequence of right anticommutativity.
\end{proposition}

\begin{proof}
There are two algebra association types, $(ab)c$ and $a(bc)$, and 12 basis monomials for $F\!A_3$,
which we list in lexicographical order:
  \[
  (ab)c, \; (ac)b, \; (ba)c, \; (bc)a, \; (ca)b, \; (cb)a, \;
  a(bc), \; a(cb), \; b(ac), \; b(ca), \; c(ab), \; c(ba).
  \]
Lemmas \ref{barsidelemma}--\ref{dialgebradimensionlemma} imply that we need only 6 dialgebra association
types:
  \[
  \begin{array}{llllll}
  ( a \dashv b ) \dashv c,
  &\quad
  ( a \vdash b ) \dashv c,
  &\quad
  ab \vdash c,
  &\quad
  a \dashv bc,
  &\quad
  a \vdash ( b \dashv c ),
  &\quad
  a \vdash ( b \vdash c ).
  \end{array}
  \]
We therefore have 36 basis monomials for $FD_3$ in lexicographical order: 
  \[
  \begin{array}{llllll}
  ( a {\dashv} b ) {\dashv} c, &\quad
  ( a {\dashv} c ) {\dashv} b, &\quad
  ( b {\dashv} a ) {\dashv} c, &\quad
  ( b {\dashv} c ) {\dashv} a, &\quad
  ( c {\dashv} a ) {\dashv} b, &\quad
  ( c {\dashv} b ) {\dashv} a,
  \\
  ( a {\vdash} b ) {\dashv} c, &\quad
  ( a {\vdash} c ) {\dashv} b, &\quad
  ( b {\vdash} a ) {\dashv} c, &\quad
  ( b {\vdash} c ) {\dashv} a, &\quad
  ( c {\vdash} a ) {\dashv} b, &\quad
  ( c {\vdash} b ) {\dashv} a,
  \\
  ab {\vdash} c, &\quad
  ac {\vdash} b, &\quad
  ba {\vdash} c, &\quad
  bc {\vdash} a, &\quad
  ca {\vdash} b, &\quad
  cb {\vdash} a,
  \\
  a {\dashv} bc, &\quad
  a {\dashv} cb, &\quad
  b {\dashv} ac, &\quad
  b {\dashv} ca, &\quad
  c {\dashv} ab, &\quad
  c {\dashv} ba,
  \\
  a {\vdash} ( b {\dashv} c ), &\quad
  a {\vdash} ( c {\dashv} b ), &\quad
  b {\vdash} ( a {\dashv} c ), &\quad
  b {\vdash} ( c {\dashv} a ), &\quad
  c {\vdash} ( a {\dashv} b ), &\quad
  c {\vdash} ( b {\dashv} a ),
  \\
  a {\vdash} ( b {\vdash} c ), &\quad
  a {\vdash} ( c {\vdash} b ), &\quad
  b {\vdash} ( a {\vdash} c ), &\quad
  b {\vdash} ( c {\vdash} a ), &\quad
  c {\vdash} ( a {\vdash} b ), &\quad
  c {\vdash} ( b {\vdash} a ).
  \end{array}
  \]
We rewrite the identities of equation \eqref{alternativedialgebraidentities}
using our basis of $FD_3$ and obtain
  \begin{equation}
  \left\{
  \begin{array}{r}
  ( a \dashv b ) \dashv c + ( b \vdash a ) \dashv c - a \dashv bc - b \vdash ( a \dashv c ) \equiv 0,
  \\
  ( a \dashv b ) \dashv c + ( a \dashv c ) \dashv b - a \dashv bc - a \dashv cb \equiv 0,
  \\
  ( a \vdash b ) \dashv c + ac \vdash b - a \vdash ( b \dashv c ) - a \vdash ( c \vdash b ) \equiv 0.
  \end{array}
  \right.
  \label{ADidentities}
  \end{equation}
Each admits six permutations, giving 18 identities which span the subspace 
of $FD_3$ consisting of the multilinear identities for alternative dialgebras. 
This subspace is the row space of an $18 \times 36$ matrix $A$; 
the rows correspond to identities and the columns to basis monomials. 
(See Figure \ref{AEtable}, with $+$, $-$, $\cdot$ for $1$, $-1$, $0$.)

The linear expansion map $E_3\colon F\!A_3 \to FD_3$ is defined on basis monomials by iteration of 
the dicommutator:
  \begin{align*}
  E_3\colon (ab)c 
  = 
  \langle \langle a,b \rangle ,c\rangle
  &\longmapsto
  ( a \dashv b ) \dashv c - ( b \vdash a ) \dashv c
  -
  c \vdash ( a \dashv b ) + c \vdash ( b \vdash a ),
  \\
  E_3\colon a(bc) 
  = 
  \langle a,\langle b,c\rangle \rangle
  &\longmapsto
  {} - bc \vdash a + cb \vdash a + a \dashv bc - a \dashv cb.
  \end{align*}
It suffices to calculate $E_3$ on one monomial of each association type; 
the other expansions are obtained by permutation.
We represent these expansions as a $12 \times 36$ matrix $E$ in which the
$(i,j)$ entry contains the coefficient of the $j$-th basis monomial of $FD_3$
in the expansion of the $i$-th basis monomial of $F\!A_3$.
(See Figure \ref{AEtable}.)

  \begin{figure}
  \begin{align*}
  A &=
  \left[
  \begin{array}{cccccccccccccccccccccccccccccccccccc}
   + &\nn \cdot &\nn \cdot &\nn \cdot &\nn \cdot &\nn \cdot &\nn
   \cdot &\nn \cdot &\nn + &\nn \cdot &\nn \cdot &\nn \cdot &\nn
   \cdot &\nn \cdot &\nn \cdot &\nn \cdot &\nn \cdot &\nn \cdot &\nn
   - &\nn \cdot &\nn \cdot &\nn \cdot &\nn \cdot &\nn \cdot &\nn
   \cdot &\nn \cdot &\nn - &\nn \cdot &\nn \cdot &\nn \cdot &\nn
   \cdot &\nn \cdot &\nn \cdot &\nn \cdot &\nn \cdot &\nn \cdot \\
   \cdot &\nn + &\nn \cdot &\nn \cdot &\nn \cdot &\nn \cdot &\nn
   \cdot &\nn \cdot &\nn \cdot &\nn \cdot &\nn + &\nn \cdot &\nn
   \cdot &\nn \cdot &\nn \cdot &\nn \cdot &\nn \cdot &\nn \cdot &\nn
   \cdot &\nn - &\nn \cdot &\nn \cdot &\nn \cdot &\nn \cdot &\nn
   \cdot &\nn \cdot &\nn \cdot &\nn \cdot &\nn - &\nn \cdot &\nn
   \cdot &\nn \cdot &\nn \cdot &\nn \cdot &\nn \cdot &\nn \cdot \\
   \cdot &\nn \cdot &\nn + &\nn \cdot &\nn \cdot &\nn \cdot &\nn
   + &\nn \cdot &\nn \cdot &\nn \cdot &\nn \cdot &\nn \cdot &\nn
   \cdot &\nn \cdot &\nn \cdot &\nn \cdot &\nn \cdot &\nn \cdot &\nn
   \cdot &\nn \cdot &\nn - &\nn \cdot &\nn \cdot &\nn \cdot &\nn
   - &\nn \cdot &\nn \cdot &\nn \cdot &\nn \cdot &\nn \cdot &\nn
   \cdot &\nn \cdot &\nn \cdot &\nn \cdot &\nn \cdot &\nn \cdot \\
   \cdot &\nn \cdot &\nn \cdot &\nn + &\nn \cdot &\nn \cdot &\nn
   \cdot &\nn \cdot &\nn \cdot &\nn \cdot &\nn \cdot &\nn + &\nn
   \cdot &\nn \cdot &\nn \cdot &\nn \cdot &\nn \cdot &\nn \cdot &\nn
   \cdot &\nn \cdot &\nn \cdot &\nn - &\nn \cdot &\nn \cdot &\nn
   \cdot &\nn \cdot &\nn \cdot &\nn \cdot &\nn \cdot &\nn - &\nn
   \cdot &\nn \cdot &\nn \cdot &\nn \cdot &\nn \cdot &\nn \cdot \\
   \cdot &\nn \cdot &\nn \cdot &\nn \cdot &\nn + &\nn \cdot &\nn
   \cdot &\nn + &\nn \cdot &\nn \cdot &\nn \cdot &\nn \cdot &\nn
   \cdot &\nn \cdot &\nn \cdot &\nn \cdot &\nn \cdot &\nn \cdot &\nn
   \cdot &\nn \cdot &\nn \cdot &\nn \cdot &\nn - &\nn \cdot &\nn
   \cdot &\nn - &\nn \cdot &\nn \cdot &\nn \cdot &\nn \cdot &\nn
   \cdot &\nn \cdot &\nn \cdot &\nn \cdot &\nn \cdot &\nn \cdot \\
   \cdot &\nn \cdot &\nn \cdot &\nn \cdot &\nn \cdot &\nn + &\nn
   \cdot &\nn \cdot &\nn \cdot &\nn + &\nn \cdot &\nn \cdot &\nn
   \cdot &\nn \cdot &\nn \cdot &\nn \cdot &\nn \cdot &\nn \cdot &\nn
   \cdot &\nn \cdot &\nn \cdot &\nn \cdot &\nn \cdot &\nn - &\nn
   \cdot &\nn \cdot &\nn \cdot &\nn - &\nn \cdot &\nn \cdot &\nn
   \cdot &\nn \cdot &\nn \cdot &\nn \cdot &\nn \cdot &\nn \cdot \\
   + &\nn + &\nn \cdot &\nn \cdot &\nn \cdot &\nn \cdot &\nn
   \cdot &\nn \cdot &\nn \cdot &\nn \cdot &\nn \cdot &\nn \cdot &\nn
   \cdot &\nn \cdot &\nn \cdot &\nn \cdot &\nn \cdot &\nn \cdot &\nn
   - &\nn - &\nn \cdot &\nn \cdot &\nn \cdot &\nn \cdot &\nn
   \cdot &\nn \cdot &\nn \cdot &\nn \cdot &\nn \cdot &\nn \cdot &\nn
   \cdot &\nn \cdot &\nn \cdot &\nn \cdot &\nn \cdot &\nn \cdot \\
   + &\nn + &\nn \cdot &\nn \cdot &\nn \cdot &\nn \cdot &\nn
   \cdot &\nn \cdot &\nn \cdot &\nn \cdot &\nn \cdot &\nn \cdot &\nn
   \cdot &\nn \cdot &\nn \cdot &\nn \cdot &\nn \cdot &\nn \cdot &\nn
   - &\nn - &\nn \cdot &\nn \cdot &\nn \cdot &\nn \cdot &\nn
   \cdot &\nn \cdot &\nn \cdot &\nn \cdot &\nn \cdot &\nn \cdot &\nn
   \cdot &\nn \cdot &\nn \cdot &\nn \cdot &\nn \cdot &\nn \cdot \\
   \cdot &\nn \cdot &\nn + &\nn + &\nn \cdot &\nn \cdot &\nn
   \cdot &\nn \cdot &\nn \cdot &\nn \cdot &\nn \cdot &\nn \cdot &\nn
   \cdot &\nn \cdot &\nn \cdot &\nn \cdot &\nn \cdot &\nn \cdot &\nn
   \cdot &\nn \cdot &\nn - &\nn - &\nn \cdot &\nn \cdot &\nn
   \cdot &\nn \cdot &\nn \cdot &\nn \cdot &\nn \cdot &\nn \cdot &\nn
   \cdot &\nn \cdot &\nn \cdot &\nn \cdot &\nn \cdot &\nn \cdot \\
   \cdot &\nn \cdot &\nn + &\nn + &\nn \cdot &\nn \cdot &\nn
   \cdot &\nn \cdot &\nn \cdot &\nn \cdot &\nn \cdot &\nn \cdot &\nn
   \cdot &\nn \cdot &\nn \cdot &\nn \cdot &\nn \cdot &\nn \cdot &\nn
   \cdot &\nn \cdot &\nn - &\nn - &\nn \cdot &\nn \cdot &\nn
   \cdot &\nn \cdot &\nn \cdot &\nn \cdot &\nn \cdot &\nn \cdot &\nn
   \cdot &\nn \cdot &\nn \cdot &\nn \cdot &\nn \cdot &\nn \cdot \\
   \cdot &\nn \cdot &\nn \cdot &\nn \cdot &\nn + &\nn + &\nn
   \cdot &\nn \cdot &\nn \cdot &\nn \cdot &\nn \cdot &\nn \cdot &\nn
   \cdot &\nn \cdot &\nn \cdot &\nn \cdot &\nn \cdot &\nn \cdot &\nn
   \cdot &\nn \cdot &\nn \cdot &\nn \cdot &\nn - &\nn - &\nn
   \cdot &\nn \cdot &\nn \cdot &\nn \cdot &\nn \cdot &\nn \cdot &\nn
   \cdot &\nn \cdot &\nn \cdot &\nn \cdot &\nn \cdot &\nn \cdot \\
   \cdot &\nn \cdot &\nn \cdot &\nn \cdot &\nn + &\nn + &\nn
   \cdot &\nn \cdot &\nn \cdot &\nn \cdot &\nn \cdot &\nn \cdot &\nn
   \cdot &\nn \cdot &\nn \cdot &\nn \cdot &\nn \cdot &\nn \cdot &\nn
   \cdot &\nn \cdot &\nn \cdot &\nn \cdot &\nn - &\nn - &\nn
   \cdot &\nn \cdot &\nn \cdot &\nn \cdot &\nn \cdot &\nn \cdot &\nn
   \cdot &\nn \cdot &\nn \cdot &\nn \cdot &\nn \cdot &\nn \cdot \\
   \cdot &\nn \cdot &\nn \cdot &\nn \cdot &\nn \cdot &\nn \cdot &\nn
   + &\nn \cdot &\nn \cdot &\nn \cdot &\nn \cdot &\nn \cdot &\nn
   \cdot &\nn + &\nn \cdot &\nn \cdot &\nn \cdot &\nn \cdot &\nn
   \cdot &\nn \cdot &\nn \cdot &\nn \cdot &\nn \cdot &\nn \cdot &\nn
   - &\nn \cdot &\nn \cdot &\nn \cdot &\nn \cdot &\nn \cdot &\nn
   \cdot &\nn - &\nn \cdot &\nn \cdot &\nn \cdot &\nn \cdot \\
   \cdot &\nn \cdot &\nn \cdot &\nn \cdot &\nn \cdot &\nn \cdot &\nn
   \cdot &\nn + &\nn \cdot &\nn \cdot &\nn \cdot &\nn \cdot &\nn
   + &\nn \cdot &\nn \cdot &\nn \cdot &\nn \cdot &\nn \cdot &\nn
   \cdot &\nn \cdot &\nn \cdot &\nn \cdot &\nn \cdot &\nn \cdot &\nn
   \cdot &\nn - &\nn \cdot &\nn \cdot &\nn \cdot &\nn \cdot &\nn
   - &\nn \cdot &\nn \cdot &\nn \cdot &\nn \cdot &\nn \cdot \\
   \cdot &\nn \cdot &\nn \cdot &\nn \cdot &\nn \cdot &\nn \cdot &\nn
   \cdot &\nn \cdot &\nn + &\nn \cdot &\nn \cdot &\nn \cdot &\nn
   \cdot &\nn \cdot &\nn \cdot &\nn + &\nn \cdot &\nn \cdot &\nn
   \cdot &\nn \cdot &\nn \cdot &\nn \cdot &\nn \cdot &\nn \cdot &\nn
   \cdot &\nn \cdot &\nn - &\nn \cdot &\nn \cdot &\nn \cdot &\nn
   \cdot &\nn \cdot &\nn \cdot &\nn - &\nn \cdot &\nn \cdot \\
   \cdot &\nn \cdot &\nn \cdot &\nn \cdot &\nn \cdot &\nn \cdot &\nn
   \cdot &\nn \cdot &\nn \cdot &\nn + &\nn \cdot &\nn \cdot &\nn
   \cdot &\nn \cdot &\nn + &\nn \cdot &\nn \cdot &\nn \cdot &\nn
   \cdot &\nn \cdot &\nn \cdot &\nn \cdot &\nn \cdot &\nn \cdot &\nn
   \cdot &\nn \cdot &\nn \cdot &\nn - &\nn \cdot &\nn \cdot &\nn
   \cdot &\nn \cdot &\nn - &\nn \cdot &\nn \cdot &\nn \cdot \\
   \cdot &\nn \cdot &\nn \cdot &\nn \cdot &\nn \cdot &\nn \cdot &\nn
   \cdot &\nn \cdot &\nn \cdot &\nn \cdot &\nn + &\nn \cdot &\nn
   \cdot &\nn \cdot &\nn \cdot &\nn \cdot &\nn \cdot &\nn + &\nn
   \cdot &\nn \cdot &\nn \cdot &\nn \cdot &\nn \cdot &\nn \cdot &\nn
   \cdot &\nn \cdot &\nn \cdot &\nn \cdot &\nn - &\nn \cdot &\nn
   \cdot &\nn \cdot &\nn \cdot &\nn \cdot &\nn \cdot &\nn - \\
   \cdot &\nn \cdot &\nn \cdot &\nn \cdot &\nn \cdot &\nn \cdot &\nn
   \cdot &\nn \cdot &\nn \cdot &\nn \cdot &\nn \cdot &\nn + &\nn
   \cdot &\nn \cdot &\nn \cdot &\nn \cdot &\nn + &\nn \cdot &\nn
   \cdot &\nn \cdot &\nn \cdot &\nn \cdot &\nn \cdot &\nn \cdot &\nn
   \cdot &\nn \cdot &\nn \cdot &\nn \cdot &\nn \cdot &\nn - &\nn
   \cdot &\nn \cdot &\nn \cdot &\nn \cdot &\nn - &\nn \cdot   \end{array}
  \right]
  \\
  \\
  E &=
  \left[
  \begin{array}{cccccccccccccccccccccccccccccccccccc}
   + &\nn \cdot &\nn \cdot &\nn \cdot &\nn \cdot &\nn \cdot &\nn
   \cdot &\nn \cdot &\nn - &\nn \cdot &\nn \cdot &\nn \cdot &\nn
   \cdot &\nn \cdot &\nn \cdot &\nn \cdot &\nn \cdot &\nn \cdot &\nn
   \cdot &\nn \cdot &\nn \cdot &\nn \cdot &\nn \cdot &\nn \cdot &\nn
   \cdot &\nn \cdot &\nn \cdot &\nn \cdot &\nn - &\nn \cdot &\nn
   \cdot &\nn \cdot &\nn \cdot &\nn \cdot &\nn \cdot &\nn + \\
   \cdot &\nn + &\nn \cdot &\nn \cdot &\nn \cdot &\nn \cdot &\nn
   \cdot &\nn \cdot &\nn \cdot &\nn \cdot &\nn - &\nn \cdot &\nn
   \cdot &\nn \cdot &\nn \cdot &\nn \cdot &\nn \cdot &\nn \cdot &\nn
   \cdot &\nn \cdot &\nn \cdot &\nn \cdot &\nn \cdot &\nn \cdot &\nn
   \cdot &\nn \cdot &\nn - &\nn \cdot &\nn \cdot &\nn \cdot &\nn
   \cdot &\nn \cdot &\nn \cdot &\nn + &\nn \cdot &\nn \cdot \\
   \cdot &\nn \cdot &\nn + &\nn \cdot &\nn \cdot &\nn \cdot &\nn
   - &\nn \cdot &\nn \cdot &\nn \cdot &\nn \cdot &\nn \cdot &\nn
   \cdot &\nn \cdot &\nn \cdot &\nn \cdot &\nn \cdot &\nn \cdot &\nn
   \cdot &\nn \cdot &\nn \cdot &\nn \cdot &\nn \cdot &\nn \cdot &\nn
   \cdot &\nn \cdot &\nn \cdot &\nn \cdot &\nn \cdot &\nn - &\nn
   \cdot &\nn \cdot &\nn \cdot &\nn \cdot &\nn + &\nn \cdot \\
   \cdot &\nn \cdot &\nn \cdot &\nn + &\nn \cdot &\nn \cdot &\nn
   \cdot &\nn \cdot &\nn \cdot &\nn \cdot &\nn \cdot &\nn - &\nn
   \cdot &\nn \cdot &\nn \cdot &\nn \cdot &\nn \cdot &\nn \cdot &\nn
   \cdot &\nn \cdot &\nn \cdot &\nn \cdot &\nn \cdot &\nn \cdot &\nn
   - &\nn \cdot &\nn \cdot &\nn \cdot &\nn \cdot &\nn \cdot &\nn
   \cdot &\nn + &\nn \cdot &\nn \cdot &\nn \cdot &\nn \cdot \\
   \cdot &\nn \cdot &\nn \cdot &\nn \cdot &\nn + &\nn \cdot &\nn
   \cdot &\nn - &\nn \cdot &\nn \cdot &\nn \cdot &\nn \cdot &\nn
   \cdot &\nn \cdot &\nn \cdot &\nn \cdot &\nn \cdot &\nn \cdot &\nn
   \cdot &\nn \cdot &\nn \cdot &\nn \cdot &\nn \cdot &\nn \cdot &\nn
   \cdot &\nn \cdot &\nn \cdot &\nn - &\nn \cdot &\nn \cdot &\nn
   \cdot &\nn \cdot &\nn + &\nn \cdot &\nn \cdot &\nn \cdot \\
   \cdot &\nn \cdot &\nn \cdot &\nn \cdot &\nn \cdot &\nn + &\nn
   \cdot &\nn \cdot &\nn \cdot &\nn - &\nn \cdot &\nn \cdot &\nn
   \cdot &\nn \cdot &\nn \cdot &\nn \cdot &\nn \cdot &\nn \cdot &\nn
   \cdot &\nn \cdot &\nn \cdot &\nn \cdot &\nn \cdot &\nn \cdot &\nn
   \cdot &\nn - &\nn \cdot &\nn \cdot &\nn \cdot &\nn \cdot &\nn
   + &\nn \cdot &\nn \cdot &\nn \cdot &\nn \cdot &\nn \cdot \\
   \cdot &\nn \cdot &\nn \cdot &\nn \cdot &\nn \cdot &\nn \cdot &\nn
   \cdot &\nn \cdot &\nn \cdot &\nn \cdot &\nn \cdot &\nn \cdot &\nn
   \cdot &\nn \cdot &\nn \cdot &\nn - &\nn \cdot &\nn + &\nn
   + &\nn - &\nn \cdot &\nn \cdot &\nn \cdot &\nn \cdot &\nn
   \cdot &\nn \cdot &\nn \cdot &\nn \cdot &\nn \cdot &\nn \cdot &\nn
   \cdot &\nn \cdot &\nn \cdot &\nn \cdot &\nn \cdot &\nn \cdot \\
   \cdot &\nn \cdot &\nn \cdot &\nn \cdot &\nn \cdot &\nn \cdot &\nn
   \cdot &\nn \cdot &\nn \cdot &\nn \cdot &\nn \cdot &\nn \cdot &\nn
   \cdot &\nn \cdot &\nn \cdot &\nn + &\nn \cdot &\nn - &\nn
   - &\nn + &\nn \cdot &\nn \cdot &\nn \cdot &\nn \cdot &\nn
   \cdot &\nn \cdot &\nn \cdot &\nn \cdot &\nn \cdot &\nn \cdot &\nn
   \cdot &\nn \cdot &\nn \cdot &\nn \cdot &\nn \cdot &\nn \cdot \\
   \cdot &\nn \cdot &\nn \cdot &\nn \cdot &\nn \cdot &\nn \cdot &\nn
   \cdot &\nn \cdot &\nn \cdot &\nn \cdot &\nn \cdot &\nn \cdot &\nn
   \cdot &\nn - &\nn \cdot &\nn \cdot &\nn + &\nn \cdot &\nn
   \cdot &\nn \cdot &\nn + &\nn - &\nn \cdot &\nn \cdot &\nn
   \cdot &\nn \cdot &\nn \cdot &\nn \cdot &\nn \cdot &\nn \cdot &\nn
   \cdot &\nn \cdot &\nn \cdot &\nn \cdot &\nn \cdot &\nn \cdot \\
   \cdot &\nn \cdot &\nn \cdot &\nn \cdot &\nn \cdot &\nn \cdot &\nn
   \cdot &\nn \cdot &\nn \cdot &\nn \cdot &\nn \cdot &\nn \cdot &\nn
   \cdot &\nn + &\nn \cdot &\nn \cdot &\nn - &\nn \cdot &\nn
   \cdot &\nn \cdot &\nn - &\nn + &\nn \cdot &\nn \cdot &\nn
   \cdot &\nn \cdot &\nn \cdot &\nn \cdot &\nn \cdot &\nn \cdot &\nn
   \cdot &\nn \cdot &\nn \cdot &\nn \cdot &\nn \cdot &\nn \cdot \\
   \cdot &\nn \cdot &\nn \cdot &\nn \cdot &\nn \cdot &\nn \cdot &\nn
   \cdot &\nn \cdot &\nn \cdot &\nn \cdot &\nn \cdot &\nn \cdot &\nn
   - &\nn \cdot &\nn + &\nn \cdot &\nn \cdot &\nn \cdot &\nn
   \cdot &\nn \cdot &\nn \cdot &\nn \cdot &\nn + &\nn - &\nn
   \cdot &\nn \cdot &\nn \cdot &\nn \cdot &\nn \cdot &\nn \cdot &\nn
   \cdot &\nn \cdot &\nn \cdot &\nn \cdot &\nn \cdot &\nn \cdot \\
   \cdot &\nn \cdot &\nn \cdot &\nn \cdot &\nn \cdot &\nn \cdot &\nn
   \cdot &\nn \cdot &\nn \cdot &\nn \cdot &\nn \cdot &\nn \cdot &\nn
   + &\nn \cdot &\nn - &\nn \cdot &\nn \cdot &\nn \cdot &\nn
   \cdot &\nn \cdot &\nn \cdot &\nn \cdot &\nn - &\nn + &\nn
   \cdot &\nn \cdot &\nn \cdot &\nn \cdot &\nn \cdot &\nn \cdot &\nn
   \cdot &\nn \cdot &\nn \cdot &\nn \cdot &\nn \cdot &\nn .
  \end{array}
  \right]
  \end{align*}
  \caption{The matrices $A$ and $E$ from the proof of Proposition \ref{degree3proposition}}
  \label{AEtable}
  \end{figure}
    
We construct a $30 \times 48$ matrix $M$; columns 1--36 correspond to the basis monomials of $FD_3$, 
and columns 37--48 to the basis monomials of $F\!A_3$.
The upper left $18 \times 36$ block is $A$;
the lower left $12 \times 36$ block is $E$;
the upper right $18 \times 12$ block is the zero matrix $O$;
the lower right $12 \times 12$ block is the identity matrix $I$:
  \begin{equation}
  \label{blockmatrix}
  M =
  \left[
  \begin{array}{r|r}
  A & O \\
  E & I
  \end{array}
  \right].
  \end{equation}
Any row of the row canonical form $\mathrm{RCF}(M)$ which has its leading 1 to the right of the vertical line 
represents a polynomial identity in $F\!A_3$ satisfied by the dicommutator in every alternative dialgebra.
These rows represent dependence relations among the basis monomials of $F\!A_3$ resulting from dependence relations, 
implied by the alternative dialgebra identities, among their expansions in $FD_3$.  
A calculation with the computer algebra system Maple shows that there are three such rows, which 
represent the three permutations of right anticommutativity: 
  \[
  a(bc) + a(cb), \qquad
  b(ac) + b(ca), \qquad
  c(ab) + c(ba).
  \] 
The restriction on the characteristic is required by the fact that 6 is the least common multiple of
the denominators of the nonzero entries of $\mathrm{RCF}(M)$.
\end{proof}

\begin{theorem} \label{degree4proposition}
Over a field of characteristic $\ne 2, 3$, every multilinear polynomial
identity in degree 4 satisfied by the dicommutator in every alternative
dialgebra is a consequence of right anticommutativity and the di-Malcev identity.
\end{theorem}

\begin{proof}
There are 20 dialgebra association types in degree 4:
  \allowdisplaybreaks
  \begin{alignat*}{4}
  &((a \dashv b) \dashv c) \dashv d,
  &\quad
  &((a \vdash b) \dashv c) \dashv d,
  &\quad
  &(ab \vdash c) \dashv d,
  &\quad
  &(a \dashv bc) \dashv d,
  \\
  &(a \vdash (b \dashv c)) \dashv d,
  &\quad
  &(a \vdash (b \vdash c)) \dashv d,
  &\quad
  &(ab)c \vdash d,
  &\quad
  &a (bc) \vdash d,
  \\
  &(a \dashv b) \dashv cd,
  &\quad
  &(a \vdash b) \dashv cd,
  &\quad
  &ab \vdash (c \dashv d),
  &\quad
  &ab \vdash (c \vdash d),
  \\
  &a \dashv (bc)d,
  &\quad
  &a \dashv b (cd),
  &\quad
  &a \vdash ((b \dashv c) \dashv d),
  &\quad
  &a \vdash ((b \vdash c) \dashv d),
  \\
  &a \vdash (bc \vdash d),
  &\quad
  &a \vdash (b \dashv cd),
  &\quad
  &a \vdash (b \vdash (c \dashv d)),
  &\quad
  &a \vdash (b \vdash (c \vdash d)).
  \end{alignat*}
Each identity $P(a,b,c)$ for alternative dialgebras 
has 10 consequences in degree 4:
  \begin{alignat*}{5}
  &P( a \vdash d, b, c ),
  &\quad
  &P( a \dashv d, b, c ),
  &\quad
  &P( a, b \vdash d, c ),
  &\quad
  &P( a, b \dashv d, c ),
  &\quad
  &P( a, b, c \vdash d ),
  \\
  &P( a, b, c \dashv d ),
  &\quad
  &P( a, b, c ) \dashv d,
  &\quad
  &P( a, b, c ) \vdash d,
  &\quad
  &d \dashv P( a, b, c ),
  &\quad
  &d \vdash P( a, b, c ).
  \end{alignat*}
We obtain the 30 identities in Table \ref{alternative4}; 
each admits 24 permutations, giving 720 elements of $FD_4$ 
which span the subspace of multilinear identities satisfied by every alternative dialgebra. 
This subspace is the row space of a $720 \times 480$ matrix $A$.

A right anticommutative operation has four association types in degree 4:
  \[
  ((ab)c)d, 
  \quad
  (a(bc))d, 
  \quad
  (ab)(cd),
  \quad
  a((bc)d),
  \] 
since $a(b(cd)) = - a((cd)b)$.  Types 2, 3 and 4 have these skew-symmetries:
  \begin{equation}
  \label{skewsymmetries}
  ( a ( c b ) ) d = - ( a ( b c ) ) d, \quad
  ( a b ) ( d c ) = - ( a b ) ( c d ), \quad
  a ( ( c b ) d ) = - a ( ( b c ) d ).
  \end{equation}
Each skew-symmetry halves the number of multilinear monomials, giving the 60 monomials of 
Table \ref{FRA4basis} which form an ordered basis of $FRA_4$, the multilinear subspace 
of degree 4 in the free right anticommutative algebra on four generators.

  \begin{table}
  \[
  \begin{array}{l}
    ( ( a \vdash d ) \dashv b ) \dashv c
  + ( b \vdash ( a \vdash d ) ) \dashv c
  - ( a \vdash d ) \dashv bc
  - b \vdash ( ( a \vdash d ) \dashv c ),
  \\
    ( ( a \dashv d ) \dashv b ) \dashv c
  + ( b \vdash ( a \dashv d ) ) \dashv c
  - ( a \dashv d ) \dashv bc
  - b \vdash ( ( a \dashv d ) \dashv c ),
  \\
    ( a \dashv bd ) \dashv c
  + ( bd \vdash a ) \dashv c
  - a \dashv ( bd ) c
  - bd \vdash ( a \dashv c ),
  \\
    ( a \dashv bd ) \dashv c
  + ( bd \vdash a ) \dashv c
  - a \dashv ( bd ) c
  - bd \vdash ( a \dashv c ),
  \\
    ( a \dashv b ) \dashv cd
  + ( b \vdash a ) \dashv cd
  - a \dashv b ( cd )
  - b \vdash ( a \dashv cd ),
  \\
    ( a \dashv b ) \dashv cd
  + ( b \vdash a ) \dashv cd
  - a \dashv b ( cd )
  - b \vdash ( a \dashv cd ),
  \\
    ( ( a \dashv b ) \dashv c ) \dashv d
  + ( ( b \vdash a ) \dashv c ) \dashv d
  - ( a \dashv bc ) \dashv d
  - ( b \vdash ( a \dashv c ) ) \dashv d,
  \\
    ( ab ) c \vdash d
  + ( ba ) c \vdash d
  - a ( bc ) \vdash d
  - b ( ac ) \vdash d,
  \\
    d \dashv ( ab ) c
  + d \dashv ( ba ) c
  - d \dashv a ( bc )
  - d \dashv b ( ac ),
  \\
    d \vdash ( ( a \dashv b ) \dashv c )
  + d \vdash ( ( b \vdash a ) \dashv c )
  - d \vdash ( a \dashv bc )
  - d \vdash ( b \vdash ( a \dashv c ) ),
  \\
    ( ( a \vdash d ) \dashv b ) \dashv c
  + ( ( a \vdash d ) \dashv c ) \dashv b
  - ( a \vdash d ) \dashv bc
  - ( a \vdash d ) \dashv cb,
  \\
    ( ( a \dashv d ) \dashv b ) \dashv c
  + ( ( a \dashv d ) \dashv c ) \dashv b
  - ( a \dashv d ) \dashv bc
  - ( a \dashv d ) \dashv cb,
  \\
    ( a \dashv bd ) \dashv c
  + ( a \dashv c ) \dashv bd
  - a \dashv ( bd ) c
  - a \dashv c ( bd ),
  \\
    ( a \dashv bd ) \dashv c
  + ( a \dashv c ) \dashv bd
  - a \dashv ( bd ) c
  - a \dashv c ( bd ),
  \\
    ( a \dashv b ) \dashv cd
  + ( a \dashv cd ) \dashv b
  - a \dashv b ( cd )
  - a \dashv ( cd ) b,
  \\
    ( a \dashv b ) \dashv cd
  + ( a \dashv cd ) \dashv b
  - a \dashv b ( cd )
  - a \dashv ( cd ) b,
  \\
    ( ( a \dashv b ) \dashv c ) \dashv d
  + ( ( a \dashv c ) \dashv b ) \dashv d
  - ( a \dashv bc ) \dashv d
  - ( a \dashv cb ) \dashv d,
  \\
    ( ab ) c \vdash d
  + ( ac ) b \vdash d
  - a ( bc ) \vdash d
  - a ( cb ) \vdash d,
  \\
    d \dashv ( ab ) c
  + d \dashv ( ac ) b
  - d \dashv a ( bc )
  - d \dashv a ( cb ),
  \\
    d \vdash ( ( a \dashv b ) \dashv c )
  + d \vdash ( ( a \dashv c ) \dashv b )
  - d \vdash ( a \dashv bc )
  - d \vdash ( a \dashv cb ),
  \\
    ( ad \vdash b ) \dashv c
  + ( ad ) c \vdash b
  - ad \vdash ( b \dashv c )
  - ad \vdash ( c \vdash b ),
  \\
    ( ad \vdash b ) \dashv c
  + ( ad ) c \vdash b
  - ad \vdash ( b \dashv c )
  - ad \vdash ( c \vdash b ),
  \\
    ( a \vdash ( b \vdash d ) ) \dashv c
  + ac \vdash ( b \vdash d )
  - a \vdash ( ( b \vdash d ) \dashv c )
  - a \vdash ( c \vdash ( b \vdash d ) ),
  \\
    ( a \vdash ( b \dashv d ) ) \dashv c
  + ac \vdash ( b \dashv d )
  - a \vdash ( ( b \dashv d ) \dashv c )
  - a \vdash ( c \vdash ( b \dashv d ) ),
  \\
    ( a \vdash b ) \dashv cd
  + a ( cd ) \vdash b
  - a \vdash ( b \dashv cd )
  - a \vdash ( cd \vdash b ),
  \\
    ( a \vdash b ) \dashv cd
  + a ( cd ) \vdash b
  - a \vdash ( b \dashv cd )
  - a \vdash ( cd \vdash b ),
  \\
    ( ( a \vdash b ) \dashv c ) \dashv d
  + ( ac \vdash b ) \dashv d
  - ( a \vdash ( b \dashv c ) ) \dashv d
  - ( a \vdash ( c \vdash b ) ) \dashv d,
  \\
    ( ab ) c \vdash d
  + ( ac ) b \vdash d
  - a ( bc ) \vdash d
  - a ( cb ) \vdash d,
  \\
    d \dashv ( ab ) c
  + d \dashv ( ac ) b
  - d \dashv a ( bc )
  - d \dashv a ( cb ),
  \\
    d \vdash ( ( a \vdash b ) \dashv c )
  + d \vdash ( ac \vdash b )
  - d \vdash ( a \vdash ( b \dashv c ) )
  - d \vdash ( a \vdash ( c \vdash b ) ).
  \end{array}
  \]
  \caption{Alternative dialgebra identities in degree 4}
  \label{alternative4}
  \end{table}

  \begin{table}
  \[
  \begin{array}{llllll}
  ((ab)c)d, &\quad 
  ((ab)d)c, &\quad 
  ((ac)b)d, &\quad 
  ((ac)d)b, &\quad
  ((ad)b)c, &\quad 
  ((ad)c)b,
  \\
  ((ba)c)d, &\quad 
  ((ba)d)c, &\quad
  ((bc)a)d, &\quad 
  ((bc)d)a, &\quad 
  ((bd)a)c, &\quad 
  ((bd)c)a, 
  \\
  ((ca)b)d, &\quad 
  ((ca)d)b, &\quad 
  ((cb)a)d, &\quad 
  ((cb)d)a, &\quad
  ((cd)a)b, &\quad 
  ((cd)b)a,  
  \\
  ((da)b)c, &\quad 
  ((da)c)b, &\quad
  ((db)a)c, &\quad 
  ((db)c)a, &\quad 
  ((dc)a)b, &\quad 
  ((dc)b)a, 
  \\
  (a(bc))d, &\quad 
  (a(bd))c, &\quad 
  (a(cd))b, &\quad 
  (b(ac))d, &\quad
  (b(ad))c, &\quad 
  (b(cd))a,  
  \\
  (c(ab))d, &\quad 
  (c(ad))b, &\quad
  (c(bd))a, &\quad 
  (d(ab))c, &\quad 
  (d(ac))b, &\quad 
  (d(bc))a, 
  \\
  (ab)(cd), &\quad 
  (ac)(bd), &\quad 
  (ad)(bc), &\quad 
  (ba)(cd), &\quad
  (bc)(ad), &\quad 
  (bd)(ac),  
  \\
  (ca)(bd), &\quad 
  (cb)(ad), &\quad
  (cd)(ab), &\quad 
  (da)(bc), &\quad 
  (db)(ac), &\quad 
  (dc)(ab), 
  \\
  a((bc)d), &\quad 
  a((bd)c), &\quad 
  a((cd)b), &\quad 
  b((ac)d), &\quad
  b((ad)c), &\quad 
  b((cd)a),  
  \\
  c((ab)d), &\quad 
  c((ad)b), &\quad
  c((bd)a), &\quad 
  d((ab)c), &\quad 
  d((ac)b), &\quad 
  d((bc)a).
  \end{array}
  \]
  \caption{Right anticommutative monomials in degree 4}
  \label{FRA4basis}
  \end{table}

The expansion map $E_4\colon FRA_4 \to FD_4$ is defined on basis monomials 
by iteration of the dicommutator. 
The result of applying $E_4$ to the first basis monomial in each association type
is displayed in Table \ref{E4results}.
The action of $E_4$ on the other basis monomials can be obtained by permutation. 
We represent these expansions as a $60 \times 480$ matrix $E$ in which
the $(i,j)$ entry contains the coefficient of the $j$-th basis monomial of
$FD_4$ in the expansion of the $i$-th basis monomial of $FRA_4$.

  \begin{table}
  \begin{alignat*}{1}
  E_4\colon 
  &
  ( ( a b ) c ) d
  =
  \langle \langle \langle a, b \rangle, c \rangle, d \rangle
  \longmapsto
  \\
  &
  ( ( a \dashv b ) \dashv c ) \dashv d
  -
  ( ( b \vdash a ) \dashv c ) \dashv d
  -
  ( c \vdash ( a \dashv b ) ) \dashv d
  +
  ( c \vdash ( b \vdash a ) ) \dashv d
  \\
  &
  -
  d \vdash ( ( a \dashv b ) \dashv c )
  +
  d \vdash ( ( b \vdash a ) \dashv c )
  +
  d \vdash ( c \vdash ( a \dashv b ) )
  -
  d \vdash ( c \vdash ( b \vdash a ) ),
  \\
  E_4\colon 
  &
  ( a ( b c ) ) d
  =
  \langle \langle a, \langle b, c \rangle \rangle, d \rangle
  \longmapsto
  \\
  &
  -
  ( bc \vdash a ) \dashv d
  +
  ( cb \vdash a ) \dashv d
  +
  ( a \dashv bc ) \dashv d
  -
  ( a \dashv cb ) \dashv d
  \\
  &
  +
  d \vdash ( bc \vdash a )
  -
  d \vdash ( cb \vdash a )
  -
  d \vdash ( a \dashv bc )
  +
  d \vdash ( a \dashv cb ),
  \\
  E_4\colon
  &
  ( a b ) ( c d )
  =
  \langle \langle a, b \rangle, \langle c, d \rangle \rangle
  \longmapsto
  \\
  &
  ( a \dashv b ) \dashv cd
  -
  ( a \dashv b ) \dashv dc
  -
  ( b \vdash a ) \dashv cd
  +
  ( b \vdash a ) \dashv dc
  \\
  &
  -
  cd \vdash ( a \dashv b )
  +
  cd \vdash ( b \vdash a )
  +
  dc \vdash ( a \dashv b )
  -
  dc \vdash ( b \vdash a ),
  \\
  E_4\colon
  &  
  a ( ( b c ) d )
  =
  \langle a, \langle \langle b, c \rangle, d \rangle \rangle
  \longmapsto
  \\
  &
  -
  ( ( b c ) d ) \vdash a
  +
  ( ( c b  ) d ) \vdash a
  +
  ( d ( b c ) ) \vdash a
  -
  ( d ( c b ) ) \vdash a
  \\
  &
  +
  a \dashv ( ( b c ) d )
  -
  a \dashv ( ( c b ) d )
  -
  a \dashv ( d ( b c ) )
  +
  a \dashv ( d ( c b ) ).
  \end{alignat*}
  \smallskip
  \caption{Equations defining the expansion map in degree 4}
  \label{E4results}
  \end{table}
 
Let $O$ denote the $720 \times 60$ zero matrix and let $I$ denote the $60 \times 60$ identity matrix. 
We combine the matrices $A$, $E$, $O$, $I$ into a $780 \times 540$ matrix $M$ as in equation \eqref{blockmatrix}.
Any row of $\mathrm{RCF}(M)$ which has 
its leading 1 to the right of column 480 represents a polynomial identity in $FRA_4$ satisfied by the 
dicommutator in every alternative dialgebra.  There are 20 such rows; the first represents the di-Malcev
identity.  Further computations show that all of these identities are linear combinations of permutations 
of the di-Malcev identity: we create a $24 \times 60$ matrix in which row $i$ contains the coefficient 
vector of the identity obtained by applying permutation $i$ to the di-Malcev identity and straightening 
the terms using right anticommutativity, and find that this matrix has rank 20. 
We did these computations using rational arithmetic with the Maple package \texttt{LinearAlgebra}.
\end{proof}


\section{Special identities for Malcev dialgebras}

A special identity (s-identity) for Malcev dialgebras is a polynomial identity 
which is satisfied by the dicommutator in every alternative dialgebra, 
but which is not a consequence of right anticommutativity and the di-Malcev identity.
In this section we describe a computational search for such identities using the representation
theory of the symmetric group; we show that there are no s-identities in degrees 5 or 6, 
so any s-identity must have degree at least 7.  

We write $R_n$ for the number of right anticommutative (RAC) association types in degree $n$.
This equals the number of right commutative association types,
for which we refer to Bremner and Peresi \cite{BP}; see also Sloane \cite{Sloane}, sequence A085748.
We write $Z_n$ for the number of 0-dialgebra association types (Lemma \ref{dialgebratypeslemma}).
Given an RAC association type in degree $n$, we apply it to the identity permutation of the variables, 
$a_1 \cdots a_n$.
We expand this monomial by interpreting each product as the dicommutator,
and obtain a linear combination of $2^{n-1}$ multilinear monomials of degree $n$ in the free 0-dialgebra.

There are three multilinear identities \eqref{ADidentities} in the definition of alternative dialgebra.
Given a multilinear dialgebra identity $I( a_1, \dots, a_n )$ of degree $n$,
we obtain $2(n+2)$ consequences of degree $n+1$: we introduce another variable $a_{n+1}$ and
consider the $2n$ substitutions obtained by replacing $a_i$ by either $a_i \dashv a_{n+1}$ or $a_i \vdash a_{n+1}$;
we also consider the four products 
  \[
  I( \,\cdots\, ) \dashv a_{n+1}, \qquad
  I( \,\cdots\, ) \vdash a_{n+1}, \qquad
  a_{n+1} \dashv I( \,\cdots\, ), \qquad
  a_{n+1} \vdash I( \,\cdots\, ).
  \]
This procedure gives $A_n = 2^{n-6} (n+1)!$ ($n \ge 6$) multilinear identities in degree $n$ 
which generate the $S_n$-module of multilinear identities for alternative dialgebras.

Let $\lambda$ be a partition of $n$ corresponding to an irreducible representation 
with dimension $d = d_\lambda$ of the symmetric group $S_n$.
Consider a matrix with $(A_n + R_n) d$ rows and $(Z_n + R_n) d$ columns,
regarded as a matrix of size $(A_n + R_n) \times (Z_n + R_n)$ in which each entry is a $d \times d$ block;
see Figure \ref{expansionmatrix}.
In the upper left part, the $d \times d$ block in position $(i,j)$ contains the representation matrix
of the terms in the $i$-th alternative dialgebra identity which have 0-dialgebra association type $j$.
In the lower left part, the $d \times d$ block in position $(i,j)$ contains the representation matrix
of the terms in the expansion of the $i$-th RAC association type which have 0-dialgebra association type $j$.
The lower right part contains the identity matrix, representing the RAC association types.

  \begin{figure}
  \[
  \left[
  \begin{array}{c|c}
  \begin{tabular}{c}
  representation matrices \\
  for alternative dialgebra \\
  identities $(A_n d \times Z_n d)$
  \end{tabular}
  &
  \begin{tabular}{c}
  zero matrix \\
  $(A_n d \times R_n d)$
  \end{tabular}
  \\
  \midrule
  \begin{tabular}{c}
  representation matrices \\
  for expansions of RAC \\
  monomials $(R_n d \times Z_n d)$  
  \end{tabular}
  &
  \begin{tabular}{c}
  identity matrix \\
  $(R_n d \times R_n d)$
  \end{tabular}
  \end{array}
  \right]
  \]
  \caption{Expansion matrix for dicommutator identities in degree $n$}
  \label{expansionmatrix}
  \end{figure} 

For each partition $\lambda$, we compute the row canonical form of this matrix, 
and identify any rows which have leading 1s to the right of the vertical line.
These rows represent linear dependence relations among the expansions of the RAC association types 
resulting from the alternative dialgebra identities; in other words, these rows represent identities
satisfied by the dicommutator in every alternative dialgebra. 
The number of these rows will be called the multiplicity of dicommutator identities for partition $\lambda$; 
see Table \ref{sidentityresults} for computational results for degrees $3 \le n \le 6$.

  \begin{table}
  \begin{center}
  \begin{tabular}{lccccccccccc}
  $n = 3$ \\
  partition $\lambda$ & 3 & 21 & $1^3$ \\
  multiplicity & 1 & 1 & 0 \\
  \midrule
  $n = 4$ \\
  partition $\lambda$ & 4 & 31 & $2^2$ & $21^2$ & $1^4$ \\
  multiplicity & 3 & 8 & 5 & 6 & 1 \\
  \midrule
  $n = 5$ \\
  partition $\lambda$ & 5 & 41 & 32 & $31^2$ & $2^21$ & $21^3$ & $1^5$ \\
  multiplicity & 8 & 31 & 38 & 43 & 35 & 25 & 5 \\
  \midrule
  $n = 6$ \\
  partition $\lambda$ & 6 & 51 & 42 & $41^2$ & $3^2$ & 321 & $31^3$ & $2^3$ & $2^21^2$ & $21^4$ & $1^6$ \\
  multiplicity & 19 & 94 & 169 & 185 & 94 & 294 & 179 & 90 & 159 & 84 & 15
  \end{tabular}
  \end{center}
  \bigskip
  \caption{Multiplicities of representations for degrees 3, 4, 5, 6}
  \label{sidentityresults}
  \end{table}
  
We perform a second computation to determine which of the dicommutator identities are
consequences of the defining identities for Malcev dialgebras.
We work in the free RAC algebra so that we only need to consider the 
consequences of the di-Malcev identity.  
If $I(a_1,\dots,a_n)$ is a multilinear nonassociative 
algebra identity in degree $n$, then we have $n+2$ consequences in degree $n+1$, obtained by
$n$ substitutions and two multiplications.
The di-Malcev identity in degree 4 therefore has 6 consequences in degree 5 and 42 consequences
in degree 6; in general we call this number $D_n$.
We also consider the skew-symmetries of the RAC association types; for an example see 
equation \eqref{skewsymmetries}.
In degrees 3, 4, 5, 6 the number of such skew-symmetries is 1, 3, 10, 28 respectively; 
in general we call this number $W_n$.

For each partition $\lambda$ we construct a matrix of size $(W_n + D_n) d \times R_n d$ 
consisting of an upper part with $W_n$ rows and $R_n$ columns of $d \times d$ blocks, 
and a lower part with $D_n$ rows and $R_n$ columns of $d \times d$ blocks.
The upper part contains the representation matrices for the skew-symmetries, and the lower part 
contains the representation matrices for the consequences of the di-Malcev identity.
We compute the row canonical form and find that in every case its rank is equal to
the multiplicity of the dicommutator identities from Table \ref{sidentityresults}.
It follows that every identity of degree less than or equal to 6, satisfied by the dicommutator in 
every alternative dialgebra, is a consequence of right anticommutativity and the di-Malcev identity.

For further information about the application of the representation theory of the symmetric group 
to polynomial identities for nonassociative algebras, see Bremner and Peresi \cite{BP},
especially Section 5.


\section{Malcev dialgebras with one or two generators} 

It is well-known that every Malcev algebra on two generators is a Lie algebra.
For dialgebras, the corresponding question is whether every two-generated Malcev dialgebra is a Leibniz algebra.
In this section we give a negative answer.

We first consider algebras on one generator.
A basis of the free right anticommutative algebra on one generator $a$ 
consists of the elements $a^n$ for $n \ge 1$ defined by $a^1 = a$ and $a^{n+1} = a^n a$; 
multiplication is determined by the equations
  \begin{equation}
  \label{structure}
  a^n a = a^{n+1},
  \qquad
  a^n a^m = 0 \; (m \ge 2).
  \end{equation}
This structure is isomorphic to the free Leibniz algebra on one generator,
since Loday and Pirashvili \cite{LodayP} have shown that the free Leibniz algebra on a set $X$ 
is linearly isomorphic to the free associative algebra on $X$.
Clearly the structure \eqref{structure} is also the free Malcev dialgebra on one generator; 
it follows that every Malcev dialgebra with one generator is a Leibniz algebra.   

Since the di-Malcev identity has degree 4, in degrees 1, 2, 3 the free Malcev dialgebra
on a set $X$ is linearly isomorphic to the free right anticommutative algebra on $X$.
For two generators $a, b$ the following 10 monomials form a basis 
of the homogeneous subspace of degree 3 in the free right anticommutative algebra:
  \[
  (aa)a, \quad    
  (aa)b, \quad   
  (ab)a, \quad   
  (ab)b, \quad   
  (ba)a, \quad   
  (ba)b, \quad   
  (bb)a, \quad   
  (bb)b, \quad   
  a(ab), \quad   
  b(ab).
  \]  
By the theorem of Loday and Pirashvili, the homogeneous subspace
of degree 3 in the free Leibniz algebra on two generators has dimension 8.
It follows that the free Malcev dialgebra with two generators is not a Leibniz algebra.   


\section{Leibniz triple systems from Malcev dialgebras} 

Loos \cite{Loos} introduced the following trilinear operation in any Malcev algebra,
  \[
  [a,b,c] = 2 (ab)c - (bc)a - (ca)b,
  \]
and proved that it satisfies the defining identities for Lie triple systems.
In this section we extend this result to the setting of dialgebras: we consider the following 
trilinear operation on any Malcev dialgebra,
  \begin{equation} \label{op}
  \langle a,b,c \rangle = 2(ab)c + a(bc) + (ac)b,
  \tag{LTP}
  \end{equation}
and prove that it satisfies the defining identities for Leibniz triple systems.
Thus any subspace of a Malcev dialgebra which is closed under this operation 
provides an example of a Leibniz triple system.

\begin{definition} \label{definitionLTS} (Bremner and S\'anchez-Ortega \cite{BSO})
A \textbf{Leibniz triple system} is a vector space with a trilinear operation $\langle -,-,- \rangle$
satisfying these identities:
  \allowdisplaybreaks
  \begin{align*}
  \langle a, \langle b, c, d \rangle, e \rangle 
  &\equiv 
  \langle \langle a, b, c \rangle, d, e \rangle 
  - 
  \langle \langle a, c, b \rangle, d, e \rangle 
  - 
  \langle \langle a, d, b \rangle, c, e \rangle 
  + 
  \langle \langle a, d, c \rangle, b, e \rangle,
  \\ 
  \langle a, b, \langle c, d, e \rangle \rangle 
  & 
  \equiv 
  \langle \langle a, b, c \rangle, d, e \rangle 
  - 
  \langle \langle a, b, d \rangle, c, e \rangle 
  - 
  \langle \langle a, b, e \rangle, c, d \rangle 
  + 
  \langle \langle a, b, e \rangle, d, c \rangle.
  \end{align*}
It follows that any monomial in the second or third association type
is equal to a linear combination of monomials in the first association type.  
The patterns of signs and permutations on the right sides of these identities correspond
to the expansions of the Lie triple products $[[b,c],d]$ and $[[c,d],e]$.
\end{definition}

\begin{lemma} \label{degree3-Malcev dialgebra}
In a Malcev dialgebra with product $ab$, the trilinear operation \eqref{op} 
does not satisfy any polynomial identity in degree $3$.
\end{lemma}

\begin{proof}
We construct an $12 \times 18$ matrix in which
the upper left $6 \times 12$ block contains the right anticommutative identities,
the lower left $6 \times 12$ block contains the expansions of the trilinear monomials,
the upper right $6 \times 6$ block contains the zero matrix,
and the lower right $6 \times 6$ block contains the identity matrix.
Columns 1--12 correspond to the 12 multilinear monomials of degree 3
in the free nonassociative algebra, 
and columns 13--18 correspond to the 6 trilinear monomials of degree 3 in the operation
\eqref{op}:
  \[
  \langle a,b,c \rangle, \quad 
  \langle a,c,b \rangle, \quad 
  \langle b,a,c \rangle, \quad 
  \langle b,c,a \rangle, \quad 
  \langle c,a,b \rangle, \quad 
  \langle c,b,a \rangle.
  \]
There are six permutations of the right anticommutative identity $a(bc) + a(cb)$;
the $(i,j)$ entry of the upper left block is the coefficient of the $j$-th nonassociative monomial
in the $i$-th permutation.
The $(6{+}i,j)$ entry of the lower left block is the coefficient of the $j$-th nonassociative monomial
in the expansion of the $i$-th trilinear monomial.
This matrix and its row canonical form are displayed in Figures \ref{degree3matrix} and \ref{degree3matrixrcf};
we have omitted the zero rows of the RCF. 
Since there is no row in the RCF which has its leading 1 in the right part of the matrix, 
there are no dependence relations among the expansions of the trilinear monomials
which hold as a result of the right anticommutative identities.
\end{proof}

  \begin{figure}
  \[
  \left[
  \begin{array}{rrrrrrrrrrrr|rrrrrr}
  . & . & . & . & . & . & 1 & 1 & . & . & . & . & . & . & . & . & . & .  \\
  . & . & . & . & . & . & 1 & 1 & . & . & . & . & . & . & . & . & . & .  \\
  . & . & . & . & . & . & . & . & 1 & 1 & . & . & . & . & . & . & . & .  \\
  . & . & . & . & . & . & . & . & 1 & 1 & . & . & . & . & . & . & . & .  \\
  . & . & . & . & . & . & . & . & . & . & 1 & 1 & . & . & . & . & . & .  \\
  . & . & . & . & . & . & . & . & . & . & 1 & 1 & . & . & . & . & . & .  \\ 
  \midrule
  2 & 1 & . & . & . & . & 1 & . & . & . & . & . & 1 & . & . & . & . & .  \\
  1 & 2 & . & . & . & . & . & 1 & . & . & . & . & . & 1 & . & . & . & .  \\
  . & . & 2 & 1 & . & . & . & . & 1 & . & . & . & . & . & 1 & . & . & .  \\
  . & . & 1 & 2 & . & . & . & . & . & 1 & . & . & . & . & . & 1 & . & .  \\
  . & . & . & . & 2 & 1 & . & . & . & . & 1 & . & . & . & . & . & 1 & .  \\
  . & . & . & . & 1 & 2 & . & . & . & . & . & 1 & . & . & . & . & . & 1  
  \end{array}
  \right]
  \]
  \caption{The $12 \times 18$ matrix from the proof of Lemma \ref{degree3-Malcev dialgebra}}
  \label{degree3matrix}
  \[
  \left[
  \begin{array}{rrrrrrrrrrrr|rrrrrr}
  1 & . & . & . & . & . & . & -1 & . &  . & . &  . &  \frac{2}{3} & -\frac{1}{3} & . & . & . & .  \\[3pt]
  . & 1 & . & . & . & . & . &  1 & . &  . & . &  . & -\frac{1}{3} &  \frac{2}{3} & . & . & . & .  \\[3pt]
  . & . & 1 & . & . & . & . &  . & . & -1 & . &  . & . & . &  \frac{2}{3} & -\frac{1}{3} & . & .  \\[3pt]
  . & . & . & 1 & . & . & . &  . & . &  1 & . &  . & . & . & -\frac{1}{3} &  \frac{2}{3} & . & .  \\[3pt]
  . & . & . & . & 1 & . & . &  . & . &  . & . & -1 & . & . & . & . &  \frac{2}{3} & -\frac{1}{3}  \\[3pt]
  . & . & . & . & . & 1 & . &  . & . &  . & . &  1 & . & . & . & . & -\frac{1}{3} &  \frac{2}{3}  \\
  . & . & . & . & . & . & 1 & 1 & . & . & . & . & . & . & . & . & . & .  \\
  . & . & . & . & . & . & . & . & 1 & 1 & . & . & . & . & . & . & . & .  \\
  . & . & . & . & . & . & . & . & . & . & 1 & 1 & . & . & . & . & . & . 
  \end{array}
  \right]
  \]
  \caption{The row canonical form of the matrix of Figure \ref{degree3matrix}}
  \label{degree3matrixrcf}
  \end{figure}
  
\begin{theorem} \label{degree5-Malcev dialgebra}
In a Malcev dialgebra with operation $ab$, every polynomial identity of degree 5 satisfied by 
the trilinear operation \eqref{op} is a consequence of the defining identities for Leibniz
triple systems (Definition \ref{definitionLTS}).
\end{theorem}

\begin{proof}
The strategy is the same as in the proof of Lemma \ref{degree3-Malcev dialgebra}, 
but the matrix is much larger and some further computations are required.
There are 14 association types for a nonassociative binary operation of degree 5:
  \begin{alignat*}{5}
  &(((ab)c)d)e, &\quad 
  &((a(bc))d)e, &\quad 
  &((ab)(cd))e, &\quad 
  &(a((bc)d))e, &\quad 
  &(a(b(cd)))e, 
  \\
  &((ab)c)(de), &\quad 
  &(a(bc))(de), &\quad
  &(ab)((cd)e), &\quad 
  &(ab)(c(de)), &\quad 
  &a(((bc)d)e), 
  \\
  &a((b(cd))e), &\quad 
  &a((bc)(de)), &\quad 
  &a(b((cd)e)), &\quad 
  &a(b(c(de))).
  \end{alignat*}
Each type admits $5!$ permutations of the variables, giving 1680 multilinear monomials
which correspond to the columns in the left part of the matrix;
we order these monomials first by association type and then lexicographically.
We need to generate the consequences of degree 5 of the defining identities for Malcev dialgebras.
A multilinear identity $I(a_1,\dots,a_n)$ of degree $n$ produces $n{+}2$ identities of degree $n{+}1$,
using $n$ substitutions and two multiplications:
  \[
  I( a_1 a_{n+1}, \dots, a_n ),
  \;
  \dots,
  \;
  I( a_1, \dots, a_n a_{n+1} ),
  \;
  I( a_1, \dots, a_n ) a_{n+1},
  \;
  a_{n+1} I( a_1, \dots, a_n ).
  \]
The right anticommutative identity of degree 3 produces 5 identities of degree 4, 
and each of these produces 6 identities of degree 5, for a total of 30. 
The di-Malcev identity produces 6 identities of degree 5. 
Altogether we have 36 identities of degree 5, and each admits $5!$ permutations, giving 4320 identities.
Hence the upper left block of the matrix $E$ in degree 5 has size $4320 \times 1680$; 
its $(i,j)$ entry is the coefficient of the $j$-th nonassociative monomial in the $i$-th multilinear identity.
There are 3 association types for a trilinear operation in degree 5:
  \[
  \langle \langle a, b, c \rangle, d, e \rangle \qquad
  \langle a, \langle b, c, d \rangle, e \rangle \qquad
  \langle a, b, \langle c, d, e \rangle \rangle.
  \]
Each type admits $5!$ permutations of the variables, giving 360 ternary monomials,
corresponding to the columns in the right part of the matrix.
The lower left block has size $360 \times 1680$; 
its $(i,j)$ entry is the coefficient of the $j$-th nonassociative monomial in the expansion, 
using equation \eqref{op}, of the $i$-th ternary monomial.
The upper right block is the $4320 \times 360$ zero matrix,
and the lower right block is the $360 \times 360$ identity matrix;
see Figure \ref{figureMDmatrix}.

  \begin{figure}
  \[
  \left[
  \begin{array}{c|c}
  \begin{tabular}{c}
  consequences in degree 5 \\
  of the Malcev dialgebra \\
  identities of Definition \ref{MDdefinition}
  \end{tabular}
  &
  \begin{tabular}{c}
  zero matrix
  \end{tabular}
  \\
  \midrule
  \begin{tabular}{c}
  expansion using operation \\
  \eqref{op} of the ternary \\
  monomials of degree 5
  \end{tabular}
  &
  \begin{tabular}{c}
  identity matrix
  \end{tabular}
  \end{array}
  \right]
  \]
  \caption{The $4680 \times 2040$ matrix from the proof of Theorem \ref{degree5-Malcev dialgebra}}
  \label{figureMDmatrix}
  \end{figure} 

We compute the row canonical form of this matrix and find that its rank is 1820.
We ignore the first 1580 rows since their leading 1s are in the left part, and 
retain only the last 240 rows which have leading 1s in the right part.
We sort these rows by increasing number of nonzero components. 
We construct another matrix with a $360 \times 360$ upper block and 
a $120 \times 360$ lower block.
For each of the identities corresponding to the last 240 rows,
we apply all $5!$ permutations of the variables, store the permuted identities in the lower block,
and compute the row canonical form; 
after each iteration, the lower block is the zero matrix.
We record the index numbers of the identities which increase the rank:
  \begin{center}
  \begin{tabular}{lrrrrrrrr}
  identity &\quad 1 &\quad 41 &\quad 71 &\quad 111 &\quad 141 &\quad 143 \\
  rank &\quad 60 &\quad 140 &\quad 160 &\quad 160 &\quad 200 &\quad 240
  \end{tabular}
  \end{center}
Further computations show that identities 1, 41, 71 and 111 
are consequences of identities 141 and 143: 
thus these two identities generate the entire 240-dimensional space. 
Identities 141 and 143 coincide, up to a permutation of the variables, 
to the identities of Definition \ref{definitionLTS}.
We used the Maple package \texttt{LinearAlgebra[Modular]}
with $p = 101$ for these computations.
\end{proof}

\begin{corollary}
Every subspace of a Malcev dialgebra closed under the trilinear operation 
$\langle a, b, c \rangle = 2(ab)c + a(bc) + (ac)b$ is a Leibniz triple system.
\end{corollary}


\section*{Acknowledgements}

The authors thank Alexander Pozhidaev for helpful comments,
and Pavel Koles\-nikov for references \cite{Chapoton} and \cite{Vallette}.
Murray Bremner was supported by a Discovery Grant from NSERC;
he thanks the Department of Algebra, Geometry and Topology at 
the University of M\'alaga for its hospitality during his visit in June and July 2011.
Luiz Peresi was supported by CNPq of Brazil.
Juana S\'anchez-Ortega was supported by the Spanish MEC and Fondos FEDER 
jointly through project MTM2010-15223, and by the Junta de Andaluc\'ia (projects FQM-336 and FQM2467).


\end{document}